\documentclass[11pt,a4paper,twoside]{article}
\usepackage[T1]{fontenc}
\usepackage[utf8]{inputenc}
\usepackage[english]{babel}
\usepackage{amssymb, amsmath, amsfonts, amsthm}
\usepackage{enumerate}
\usepackage{colonequals}
\usepackage{empheq,fancybox}
\usepackage[small]{titlesec}
\usepackage[noblocks]{authblk}

\usepackage{microtype}
\usepackage{ellipsis}
\usepackage[BCOR=20mm,DIV=11]{typearea}
 \newcommand{\hm}[1]{\leavevmode{\marginpar{\tiny%
 $ \hbox to 0mm{\hspace*{-0.5mm} $ \leftarrow $ \hss}%
 \vcenter{\vrule depth 0.1mm height 0.1mm width \the\marginparwidth}%
 \hbox to
 0mm{\hss $ \rightarrow $ \hspace*{-0.5mm}} $ \\\relax\raggedright #1}}}

\newcommand{\euler}{\mathrm{e}}
\newcommand{\drm}{\mathrm{d}}
\newcommand{\dvol}{\mathrm{dvol}}
\newcommand{\RR}{\mathbb{R}}

\newcommand{\NN}{\mathbb{N}}

\newcommand{\tvert}[1]{{\left\vert\kern-0.25ex\left\vert\kern-0.25ex\left\vert #1 
    \right\vert\kern-0.25ex\right\vert\kern-0.25ex\right\vert}}
\renewcommand{\epsilon}{\varepsilon}

\DeclareMathOperator{\diam}{\mathop{diam}}

\DeclareMathOperator{\Tr}{\mathop{Tr}}

\DeclareMathOperator{\Ric}{\mathop{Ric}}
\DeclareMathOperator{\Vol}{\mathop{Vol}}

\newtheorem{theorem}{Theorem}[section]

\newtheorem{proposition}[theorem]{Proposition}
\newtheorem{corollary}[theorem]{Corollary}
\theoremstyle{definition}

\theoremstyle{remark}

\usepackage{color}
	\definecolor{darkred}{rgb}{0.5,0,0}
	\definecolor{darkgreen}{rgb}{0,0.5,0}
	\definecolor{darkblue}{rgb}{0,0,0.5}
\begin{document}
\title{Li-{Y}au gradient estimate for compact manifolds with negative part of Ricci curvature in the Kato class}
\author{Christian Rose}
%
%
\affil{Technische Universit\"at Chemnitz, Faculty of Mathematics, D - 09107 Chemnitz}
\date{\today}
\maketitle
\section{Introduction}
Heat kernel estimates are of particular interest in geometric analysis. Explicit bounds lead to estimates of certain geometric invariants, such as eigenvalue bounds for the Laplace-Beltrami operator or Betti numbers on compact manifolds. The dependence of heat kernel bounds on the underlying geometry is crucial. \cite{LiYau-86} shows that it is possible to derive such estimates using gradient estimates for solutions of the heat equation. Several authors generalized the assumption of a sharp lower Ricci curvature bound like in \cite{LiYau-86} and derived explicit heat kernel estimates. Gallot obtained in \cite{Gallot-88} the first bounds depending on global $L^p$-Ricci curvature estimates and derived also eigenvalue and Betti number bounds. In \cite{Rose-16} the author proves a heat kernel estimate for locally uniformly Ricci curvature $L^p$-bounds and shows that under these conditions, the first Betti number of a compact manifold can also be bounded explicitly. \\
From the analytic point of view, $L^p$-conditions on the negative part of the Ricci curvature operator are a good, but not the optimal curvature condition for which one can expect bounds on analytic, and in turn topological, objects on a manifold. The natural generalization comes with the Kato condition from potential theory. A measurable non-negative function $V\colon M\to\RR$ is said to be in the Kato class $\mathcal{K}(M)$, if there is a $\beta>0$ such that 
\[
b_{\rm{Kato}}(\beta, V):=\int_0^\beta \Vert P_tV\Vert_\infty\drm t<1\text,
\]
where $(P_t)_{t\geq 0}$ denotes the heat semigroup associated to the Laplace-Beltrami operator $\Delta\geq 0$ acting on functions on $M$. If the negative part of the Ricci curvature $\rho_-$ satisfies a $L^p$-bound which is small enough, then the heat kernel can be controlled and therefore also $b_{\rm{Kato}}(\beta, \rho_-)$. This article shows that it the assumption that the negative part of the Ricci curvature is in the Kato class implies a gradient estimate for solutions of the heat equation. This shows that a lot of geometric invariants can be controlled intrinsically without the exact knowledge about the behavior of the heat kernel.
The obtained Harnack inequality gives a heat kernel upper bound using the standard argument in \cite{LiYau-86}, and following the arguments in  \cite{RoseStollmann-15} our results imply purely analytic bounds for the first Betti number in terms of the Kato condition.\\
\emph{I want to thank my PhD advisor Peter Stollmann for fruitful discussions about the topic and Alexander Grigor'yan for raising the question at the conference \grqq{}Heat kernels and Analysis on Manifolds and Fractals\grqq{} in Bielefeld. Furthermore I want to thank Gilles Carron for pointing at the paper \cite{ZhangZhu-15}.
}
\section{Gradient estimates under Ricci curvature Kato condition}
Consider a closed Riemannian manifold $M$ of dimension $n\in\NN$. Its Ricci curvature tensor $\Ric\colon T^*M\to T^*M$ is a continuous function on $M$ and maps as an endomorphism from the tangent space at any point into itself and is pointwise a symmetric matrix. The function $\rho\colon M\to\RR$ maps any point $x\in M$ to the smallest eigenvalue of $\Ric(x)$. \\
For any $t\geq 0$, the operator $P_t\colon L^2(M)\to L^2(M)$ is given by $P_t=\euler^{-t\Delta}$. Our first result is a variant of the famous gradient bounds of \cite{LiYau-86} based on the strategy of \cite{ZhangZhu-15}.
\begin{theorem}\label{gradientbound}
Let $M$ be a compact Riemannian manifold of dimension $n\in\NN$, $n\geq 2$ and let $u$ be a positive solution of the heat equation
\begin{align}\label{heatequation}
\partial_tu=-\Delta u\text.
\end{align}
For any $\alpha\in(0,1)$ 
set $\delta:=\frac{2(1-\alpha)^2}{n+(1-\alpha)^2}$. 
Assume that there exists a 
$\beta>0$ such that $$b:=b_{\rm{Kato}}(\beta,\rho_-)<\frac{\delta}{5-\delta}\text.$$
Then we have
\begin{align}\label{gradientestimate}
\alpha\, j(t)\frac{\vert \nabla u\vert^2}{u^2}-\frac{\partial_t u}{u}\leq \frac{n}{(2-\delta)\alpha j(t)}\frac 1t\text,
\end{align}
	for all $t\in (0,\infty)$, where 
	\begin{align}\label{boundsonjexplicit}
	j(t)=
	(1-b)^{\left(1+\frac t\beta\right)\frac{\delta}{5-\delta}}\text.
	\end{align}
\end{theorem}
The above result can be proven following the proof of Theorem 1.1 in \cite{ZhangZhu-15}. The important step consists of controlling the solution of the problem below.
\begin{proposition}\label{controlfunction}
Let $M$ be a closed Riemannian manifold of dimension $n\in\NN$, $n\geq 2$, $\delta<5$ and consider the problem
\begin{align}\label{differentialeq}
\begin{cases}
-\Delta J-2\rho_-J -5\delta^{-1}\frac{\vert\nabla J\vert^2}{J}-\partial_tJ=0&\text{on $M\times(0,\infty)$,}\\
J(\cdot,0)=1\text.&
\end{cases}
\end{align}
Assume that $b:=b_{\rm{Kato}}(\beta,\rho_-)<(5\delta^{-1}-1)^{-1}$. Then \eqref{differentialeq} has a smooth unique solution satisfying
\begin{align}\label{boundsonJ}
j(t)\leq J(x,t)\leq 1\text,
\end{align}
for all $t\geq 0$ and $x\in M$, where $j(t)$ is given by \eqref{boundsonjexplicit}.
\end{proposition}
\begin{proof}
Let $a:=5\delta^{-1}$ and $$w=J^{-(a-1)}\text.$$ An easy calculation shows that $w$ satisfies
\begin{align}\label{transformationeq}
\begin{cases}
-\Delta w-\partial_t w +2(a-1)\rho_-w=0 & \text{on $M\times (0,\infty)$,}\\
w(\cdot,0)=1\text.&
\end{cases}
\end{align}
Since $\rho_-$ is continuous, it is relatively bounded with respect to $\Delta$, such that the semigroup $\euler^{-t(\Delta-2(a-1)\rho_-)}$ is continuous from $L^p(M)$ to $L^q(M)$ for all $1\leq p\leq q\leq \infty$ for all $t>0$, see \cite{Voigt-86}.
Therefore, the solution of \eqref{transformationeq} is given by
$$w=\euler^{-t(\Delta-2(a-1)\rho_-)}1\text.$$
The Trotter product formula and stochastic completeness imply $w\geq 1$. 
For the upper bound on $w$ we use Proposition 5.3 from \cite{RoseStollmann-15}, giving 
$$\Vert\euler^{-t(\Delta-2(a-1)\rho_-)}\Vert_{1,1}\leq \frac 1{1-b}\euler^{t\frac 1\beta \log\frac 1{1-b}}\text.$$
By duality, we have
\begin{align}
\Vert w\Vert_\infty\leq \Vert\euler^{-t(\Delta-2(a-1)\rho_-)}\Vert_{\infty,\infty}=\Vert\euler^{-t(\Delta-2(a-1)\rho_-)}\Vert_{1,1}\text.
\end{align}
The definition of $w$ now implies the stated bounds of the function $J$.
\end{proof}
\begin{proof}[Proof of Theorem \ref{gradientbound}]
We follow the steps in \cite{ZhangZhu-15}. Let $J=J(x,t)$ be a smooth positive function and 
$$Q(x,t)=\alpha J\frac{\vert\nabla u\vert^2}{u^2}-\frac{\partial_t u}{u}$$ 
and derive the inequality
\begin{align}\label{firstcomputation}
(-\Delta-\partial_t)(tQ)+2\frac{\nabla u}{u}\nabla (tQ)
&\geq \alpha t\frac{2-\delta}n J(\vert\nabla f\vert^2-\partial_tf)^2 \nonumber \\
&\quad+\alpha\left[-\Delta J-2\rho_-J-5\delta^{-1}\frac{\vert\nabla J\vert^2}{J}-\partial_tJ\right]t\vert \nabla f\vert^2\nonumber\\ 
&\quad-\delta\alpha tJ\vert \nabla f\vert^4-Q\text,
\end{align}
where $f=\ln u$. Propostion \ref{controlfunction} gives 
\begin{align}\label{firstcomputation}
(-\Delta-\partial_t)(tQ)+2\frac{\nabla u}{u}\nabla (tQ)
\geq \alpha t\frac{2-\delta}n J(\vert\nabla f\vert^2-\partial_tf)^2 
-\delta\alpha tJ\vert \nabla f\vert^4-Q\text,
\end{align}
For $T>0$, let $(x_0,t_0)$ be a maximum point of $tQ$ in $M\times [0,T]$. 
At this point the above inequality implies 
\begin{align*}
0\geq \alpha t(2J-\delta J)\frac 1n \left(\vert \nabla f\vert^2-\partial_t f\right)^2-\delta\alpha tJ\vert\nabla f\vert^4-Q\text.
\end{align*}
W.~l.~o.~g. $Q(x_0,t_0)\geq 0$. Then
\begin{align*}
\left(\vert\nabla f\vert^2-\partial f\right)^2\geq \left(\alpha J\frac{\vert\nabla u\vert^2}{u^2}-\frac{\partial_t u}{u}\right)^2+ (1-\alpha J)^2\vert\nabla f\vert^4\text. 
\end{align*}
Plugging this into the previous inequality gives 
\begin{align*}
0\geq \alpha \frac{2-\delta}{n}JtQ^2+\left(\frac{2-\delta}{n}(1-\alpha J)^2-\delta\right)\alpha tJ\vert\nabla f\vert^4-Q\text.
\end{align*}
The choice of $\delta$ implies 
\[
\frac{2-\delta}n(1-\alpha)^2-\delta=0
\]
and $J\leq 1$ gives 
\[
\frac{2-\delta}n(1-\alpha J)^2-\delta\geq 0 \quad\text{on}M\times [0,\infty)\text.
\]
At $(x_0,t_0)$ we have 
\[
0\geq \alpha \frac{2-\delta}nJt^2Q^2-tQ
\]
and therefore
\[
tQ\leq tQ\big|_{(x_0,t_0)}\leq \frac n{(2-\delta)\alpha j(t)},
\]
what implies the assertion.
\end{proof}
\section{Harnack inequality and heat kernel bounds}
Following the standard argument in \cite{LiYau-86} we obtain Harnack inequalities and heat kernel bounds. The geodesic distance on $M$ will be denoted by $d(\cdot,\cdot)$.
\begin{theorem}\label{theoremharnack}
Let $M$ be a closed Riemannian manifold of dimension $n\in\NN$, $n\geq 2$. For 
$\alpha\in(0,1)$ let $\beta>0$ and $\delta>0$ as in Theorem \ref{gradientbound} such that $$b:=b_{\rm{Kato}}(\beta,\rho_-)<\frac{\delta}{5-\delta}\text.$$ For any $T>0$, any positive solution of \eqref{heatequation} satisfies for all $x,y\in M$ and $0<t_1<t_2\leq T$
\begin{align}\label{Harnack1}
u(x,t_1)\leq u(y,t_2)\left(\frac{t_2}{t_1}\right)^{\frac n{(2-\delta)\alpha}\left(\frac 1{1-b}\right)^{\left(1+\frac {T}\beta\right)\frac\delta{5-\delta}}}\exp\left(\left(\frac 1{1-b}\right)^{\left(1+\frac {T}\beta\right)\frac\delta{5-\delta}}\frac{d(x,y)^2}{4(t_2-t_1)\alpha}\right)\text.
\end{align}
%
%
\end{theorem}
\begin{proof}
We follow  the proof of Theorem 2.1 of \cite{LiYau-86}. Let $\gamma\colon[0,1]\to M$ be any curve with $\gamma(0)=y$ and $\gamma(1)=x$ and define 
$$\eta\colon[0,1]\to M\times[t_1,t_2],\quad s\mapsto (\gamma(s),(1-s)t_2+st_1)\text.$$
Integrating $\frac d{ds}\log u$ along $\eta$, we get
\begin{align}
\log u(x,t_1)-\log u(y,t_2)=\int_0^1\langle\dot\gamma,\nabla(\log u)\rangle-(t_2-t_1)(\log u)_t\drm s\text.
\end{align}
Applying Theorem \ref{gradientbound} to $-(\log u)_t$ yields
\begin{align}\label{pathintegral}
\log\frac{u(x,t_1)}{u(y,t_2)}\leq \int_0^1\vert \dot\gamma\vert\vert\nabla\log u\vert-(t_2-t_1)\alpha j(t)\vert\nabla\log u\vert^2+(t_2-t_1)\frac n{(2-\delta)\alpha j(t)t}\drm s\text.
\end{align}
Viewing $\vert \nabla\log u\vert$ as a variable and the integrand as a quadratic in it, we observe that the integrand can be bounded by 
$$\frac{\vert\dot\gamma\vert^2}{4(t_2-t_1)\alpha j(t)}\text.$$
By the definition of $t=(1-s)t_2+st_1$
\begin{align}\label{maximumj}
\max_{s\in[0,1]}\frac 1{j(t)}\leq \max_{s\in[0,1]} \left(\frac 1{1-b}\right)^{\left(1+\frac t\beta\right)\frac\delta{5-\delta}}\leq \left(\frac 1{1-b}\right)^{\left(1+\frac {T}\beta\right)\frac\delta{5-\delta}}\text.
\end{align}
Furthermore, we have
\begin{align*}
\int_0^1\frac{t_2-t_1}{(1-s)t_2+st_1}\drm s=\log\frac{t_2}{t_1}\text.
\end{align*}
Using \eqref{maximumj} we can turn \eqref{pathintegral} into
\begin{align*}
\log\frac{u(x,t_1)}{u(y,t_2)}&\leq\left(\frac 1{1-b}\right)^{\left(1+\frac {T}\beta\right)\frac\delta{5-\delta}}\left\{\frac 1{4(t_2-t_1)\alpha}\int_0^1\vert\dot\gamma\vert^2\drm s
+\frac n{(2-\delta)\alpha}\log\frac{t_2}{t_1}\right\}\\
\end{align*}
Taking exponentials yields the claim.
\end{proof}
%
%
%
\begin{theorem}
Let $M$ be a closed Riemannian manifold of dimension $n\in\NN$, $n\geq 2$. Assume that there are constants $\alpha\in(0,1)$, $\beta>0$ and $\delta>0$ as in Theorem \ref{gradientbound} such that $$b:=b_{\rm{Kato}}(\beta,\rho_-)<\frac{\delta}{5-\delta}\text.$$
There exists an explicit $C_1=C_1(n,\alpha,\delta,\beta,b,\diam M)>0$ such that for any $t\in(0,\beta/2)$ and $x\in M$ the heat kernel $p_t(\cdot,\cdot)$ of $M$ is bounded from above by
\begin{align}\label{upperboundondiag}
p_t(x,x)\leq \frac{C_1}{\Vol(M)}t^{-\frac n{(2-\delta)\alpha}(1-b)^{-\frac 12}}\text.
\end{align}
In particular, such that for all $x,y\in M$ and $t\in (0,\beta/2)$
\begin{align}\label{upperboundoffdiag}
p_t(x,y)\leq \frac{C_1}{\Vol(M)}t^{-\frac n{(2-\delta)\alpha}(1-b)^{-\frac 12}}\exp\left(\frac{\diam(M)^2}t-\frac{d(x,y)^2}{4t}\right)\text.
\end{align}
\end{theorem}
\begin{proof}
Fix $T\leq \beta$.
Notice that $$\frac \delta{5-\delta}\leq \frac 14$$ and therefore $$\left(\frac 1{1-b}\right)^{\left(1+\frac {T}\beta\right)\frac\delta{5-\delta}}\leq \left(\frac 1{1-b}\right)^{\frac 12}\text.$$
Now let $x\in M$, $t<T$ and $s>0$ such that $t+s\leq T$. Choose $u(\cdot, t)=p_t(x,\cdot)$. Inserting the two estimates above in inequality \eqref{Harnack1} implies for all $y\in M$
\begin{align}\label{upperbound1}
p_t(x,x)\leq p_{t+s}(x,y)\left(\frac{t+s}{t}\right)^{\frac n{(2-\delta)\alpha}\left(\frac 1{1-b}\right)^{\frac 12}}\exp\left(\left(\frac 1{1-b}\right)^{\frac 12}\frac{d(x,y)^2}{4s\alpha}\right)\text.
\end{align}
Integrating \eqref{upperbound1} with respect to the $y$-variable, using $\int_Mp_t(x,y)\dvol(y)= 1$ for all $x\in M$ and all $t>0$ and $t+s\leq \beta$ yields
\begin{align}\label{upperbound2}
\Vol(M)p_t(x,x)\leq\left(\frac{\beta}{t}\right)^{\frac n{(2-\delta)\alpha}\left(\frac 1{1-b}\right)^{\frac 12}}\exp\left(\left(\frac 1{1-b}\right)^{\frac 12}\frac{\diam(M)^2}{4s\alpha}\right)\text.
\end{align}
To eliminate the parameter $s$ on the right hand side, sjnce $s\in(0,\beta-t]$ we can choose $s:= (\beta-t)/2$ such that 
\begin{align}\label{upperbound3}
\Vol(M)p_t(x,x)\leq\left(\frac{\beta}{t}\right)^{\frac n{(2-\delta)\alpha}\left(\frac 1{1-b}\right)^{\frac 12}}\exp\left(\left(\frac 1{1-b}\right)^{\frac 12}\frac{\diam(M)^2}{2(\beta-t)\alpha}\right)\text.
\end{align}
By assumption $t\leq \beta/2$, yielding the statement about the on-diagonal bound with 
\[
C_1:=\beta^{\frac n{(2-\delta)\alpha}\left(\frac 1{1-b}\right)^{\frac 12}}\exp\left(\left(\frac 1{1-b}\right)^{\frac 12}\frac{\diam(M)^2}{2(\beta-t)\alpha}\right)\text.
\] 
The on-diagonal bound self-improves directly to the off-diagonal bound by Theorem 15.13 in \cite{Grigoryan-09} setting $D=2$ in the notation of the cited book.
\end{proof}
\section{Bounds on $b_1(M)$}
Peter Stollmann and the author showed in \cite{RoseStollmann-15} that controlling the first Betti number $b_1(M)$ is closely related to semigroup perturbations by Kato-type potentials. It turns out that the only quantity which has to be controlled is the operator norm $\Vert\euler^{-t(\Delta+\rho)}\Vert_{1,\infty}$, where $\Vert\cdot\Vert_{p,q}$ denotes the operator norm form $L^p(M)$ to $L^q(M)$. The knowledge about the behavior of the unperturbed semigroup $\euler^{-t\Delta}$ is in fact sufficient to control the semigroup perturbed by a Kato-potential. Since we proved that the Kato-condition on the negative part of Ricci curvature controls the unperturbed semigroup, this condition is optimal. That means that from the perturbation theoretic point of view, nothing better can be expected.

\begin{corollary}
Let $M$ be a compact Riemannian manifold of dimension $n\geq 3$ and let $\alpha, \delta,\beta$ and $b$ as in Theorem \ref{gradientbound}. Then there is an explicit constant $B$ depending only on these quantities such that
\begin{align}
b_1(M)\leq B\text.
\end{align}
\end{corollary}
\begin{proof}
The proof of Corollary 5.7 in \cite{RoseStollmann-15} tells us 
\begin{align}
b_1(M)=\dim (H^1(M))\leq \Tr(\euler^{-t\Delta^1})\leq n\Tr(\euler^{-t(\Delta+\rho)})\quad \text{for } t>0\text.
\end{align}
Especially
$$\Tr(\euler^{-\beta/2(\Delta+\rho)})=\int_Mk(x,x)\dvol(x)$$
with a continuous kernel $k$ (note that $\rho$ is continuous) that can be estimated pointwise by
$$0\leq \sup_{x,y\in M}k(x,y)\leq \Vert\euler^{-\beta/2(\Delta+\rho)}\Vert_{1,\infty}\text,$$
giving 
$$b_1(M)\leq \Vol(M)\Vert\euler^{-\beta/2(\Delta+\rho)}\Vert_{1,\infty}\text.$$
Following the proof of Corollary 5.4 in \cite{RoseStollmann-15}, we see that 
\begin{align*}
\Vert\euler^{-\beta/2(\Delta+\rho)}\Vert_{1,\infty}&\leq \frac{C_1}{\Vol(M)}\left(\frac 2{1-b}\right)^{\frac 32\frac{1+b}{1-b}+\frac n{(2-\delta)\alpha}\left(\frac 1{1-b}\right)^{\frac 12}}\cdot\ldots\\
&\quad\ldots\cdot\left(\frac 2\beta\right)^{\frac n{(2-\delta)\alpha}\left(\frac 1{1-b}\right)^{\frac 12}}\exp\left(\frac{\beta\diam(M)^2(1-b)}{2(1+b)}\right)\text.
\end{align*}
\end{proof}
%
%
%
%
%
%
%


\end{document}